\documentclass[11pt]{amsart}
\usepackage{geometry}
\usepackage{amsfonts}
\usepackage{amsmath}
\usepackage{amssymb}
\usepackage{amsthm}
\usepackage{mathtools}
\usepackage{enumitem}
\usepackage{hyperref}
\usepackage{url}
\usepackage{color,xcolor}
\usepackage[all]{xy}
\usepackage{bbm}
\usepackage{tikz}

\newtheorem{theorem}{Theorem}[section]
\numberwithin{equation}{section}
\newtheorem{lemma}[theorem]{Lemma}
\newtheorem{definition}[theorem]{Definition}
\newtheorem{notion}[theorem]{}

\newtheorem{proposition}[theorem]{Proposition}
\newtheorem{corollary}[theorem]{Corollary}
\theoremstyle{definition}
\newtheorem{remark}[theorem]{Remark}

\newcommand{\K}{\mathrm{K}}

\newcommand{\Aut}{\operatorname{Aut}}

\newcommand{\Prim}{\operatorname{Prim}}

\newcommand{\id}{\operatorname{id}}

\title[The Necessity of Coefficient Transformations]
{Bockstein operations and AD algebras with  unbounded torsion in ${\rm K}_1$}

\author{Qingnan An}
\address{School of Mathematics and Statistics, Northeast Normal University, Changchun, {\rm130024}, China}
\email{qingnanan1024@outlook.com}

\author{Zhichao Liu}
\address{School of Mathematical Sciences,
Dalian University of Technology,
Dalian, {\rm 116024}, China }
\email{lzc.12@outlook.com}

\author{Xin Ma}
\address{Institute for Advanced Study in Mathematics, Harbin Institute of Technology, Harbin, China, {\rm 150001}}
\email{xma17@hit.edu.cn}

\subjclass[2020]{Primary 46L35; Secondary 19K35, 46L80}

\keywords{AD algebra,
Total K-theory,
Bockstein operations, Classification}

\begin{document}

\begin{abstract}
Eilers showed that for AD algebras of real rank zero with bounded torsion in $\mathrm{K}_1$, the coefficient transformations $\kappa$ are redundant in the classification by ordered scaled total $K$-theory. In this paper we treat the unbounded torsion case and prove that, in contrast, $\kappa$ becomes necessary. Specifically, we construct two non-isomorphic unital AD algebras of real rank zero, $E_0$ and $E_1$, such that their ordered scaled total $K$-theory invariants agree when the $\kappa$-maps are forgotten, i.e.,
\[ \bigl( \underline{\mathrm{K}}(E_0), \underline{\mathrm{K}}(E_0)_+, [1_{E_0}] \bigr)_{\underline{\mathrm{K}}_{\langle\kappa\rangle}} \cong \bigl( \underline{\mathrm{K}}(E_1), \underline{\mathrm{K}}(E_1)_+, [1_{E_1}] \bigr)_{\underline{\mathrm{K}}_{\langle\kappa\rangle}} \]
but are not isomorphic under the full $\Lambda$-module structure:
\[ \bigl( \underline{\mathrm{K}}(E_0), \underline{\mathrm{K}}(E_0)_+, [1_{E_0}] \bigr)_{\Lambda} \not\cong \bigl( \underline{\mathrm{K}}(E_1), \underline{\mathrm{K}}(E_1)_+, [1_{E_1}] \bigr)_{\Lambda} .\]
This completes the picture for the necessity of all three operations $\rho$, $\beta$, and $\kappa$ in this context.
\end{abstract}

\maketitle

\section{Introduction}

The classification of nuclear $C^*$-algebras by K-theoretic invariants requires more than the ordered ordinary K-groups alone. For large classes of  non-simple $C^*$-algebras, including AD algebras of real rank zero, the appropriate invariant is the total K-theory:
$$
 \underline{\rm K}(A)
 =
 \K_0(A)\oplus\K_1(A)
 \oplus
 \bigoplus_{n\geq 2}
 \left(
   \K_*(A;\mathbb Z_n)
   \oplus
   \K_1(A;\mathbb Z_n)
 \right),
$$
This invariant contains the natural coefficient transformations
$$
 \rho_n^i:
 \K_i(A)\longrightarrow\K_i(A;\mathbb Z_n),\quad
  \kappa_{m,n}^i:
 \K_i(A;\mathbb Z_n)
 \longrightarrow
 \K_i(A;\mathbb Z_{m})
$$
and the Bockstein maps
$$
 \beta_n^i:
 \K_i(A;\mathbb Z_n)\longrightarrow\K_{i+1}(A).
$$

The category of all these maps and their compositions is denoted by $\Lambda$. The total ${\rm K}$-theory  also encodes an order structure. If $I$ is an ideal of $A$, the image
$\underline{\mathrm{K}}(I\mid A)$ of $\underline{\mathrm{K}}(I)$ in $\underline{\mathrm{K}}(A)$
carries information about the position of $I$ inside $A$.
To capture the order structure coming from projections, Dadarlat and Gong
\cite[Definition 4.6]{DadarlatGong1997} introduced a positive cone
$\underline{\mathrm{K}}(A)_+$. For a separable ${\rm C}^*$-algebra of stable rank one, denote by
\begin{align*}
\underline{\mathrm{K}}(A)_+ = \Bigl\{ \bigl([e],\mathfrak{u},\bigoplus_{n\ge1}(\mathfrak{s}_{n,0},\mathfrak{s}_{n,1})\bigr) \Bigm|
& [e]\in\mathrm{K}_0^+(A), \\
\bigl([e],\mathfrak{u},\bigoplus_{n\ge1}
(\mathfrak{s}_{n,0},\mathfrak{s}_{n,1})\bigr)
&\in\underline{\mathrm{K}}(I_e\mid A)
\Bigr\},
\end{align*}
where $I_e$ is the ideal generated by a projection representing $[e]$ and
$\underline{\mathrm{K}}(I_e\mid A)$ is the image of $\underline{\mathrm{K}}(I_e)$ in
$\underline{\mathrm{K}}(A)$. (In particular, we may also denote
$\mathrm{K}_*^+(A):=\mathrm{K}_*(A)\cap \underline{\mathrm{K}}(A)_+,$
where $\mathrm{K}_*(A)$ is identified with its natural image in $\underline{\mathrm{K}}(A)$.)

\begin{definition}\rm The total K-theory for a ${\rm C}^*$-algebra $A$ as an invariant is the triple $(\underline{\mathrm{K}}(A),\underline{\mathrm{K}}(A)_+,\Sigma A)$
consisting of the total ${\rm K}$-theory group, its positive cone, and the scale $\Sigma A$
(the image of the set of projections in $A$ under the natural map into
$\mathrm{K}_0(A)$). If $A$ is unital, we will take the distinguished element
$[1_A]$ instead of  $\Sigma A$.
\end{definition}

Total $K$-theory has proven to provide a more effective description of $C^*$-algebras of real rank zero. Eilers (\cite{Eiphd, Eilers1996}) showed that $\underline{\mathrm{K}}(A)$ serves as a complete invariant for AD algebras with real rank zero and bounded torsion in $\mathrm{K}_1$. Subsequently, Dadarlat and Gong \cite{DadarlatGong1997} established an important classification for AHD algebras of real rank zero. This naturally raises the question of whether all operations in the $\Lambda$-module structure---namely, $\rho$, $\beta$, and $\kappa$---are necessary invariants.

Regarding this question, the necessity of the Bockstein map $\beta$ was established by Dadarlat and Eilers in \cite{DadarlatEilers1999}. They constructed two non-isomorphic AD algebras of real rank zero that become indistinguishable when the operation $\beta$ is omitted. More recently, the necessity of $\rho$ was demonstrated in \cite{AL}, following a similar approach.

As for the remaining operation $\kappa$, Eilers \cite{Eilers1996} established that, for AD algebras of real rank zero with bounded torsion $K_1$-groups, the coefficient transformations $\kappa$ are redundant in the classification. Whether this redundancy persists for algebras with unbounded torsion $K_1$-groups---or in broader settings such as AHD algebras---has remained an open (see \cite{Eiphd, Eilers1996}). The present paper settles this issue definitively. In stark contrast to the bounded torsion case, we prove that for AD algebras of real rank zero with unbounded torsion $K_1$-groups, the Bockstein map $\kappa$ is in fact necessary. This result reveals a fundamental dichotomy between the two torsion regimes, which further indicates the indispensable role of $\kappa$ in the general classification theory.

\begin{definition}\rm
Denote the resulting invariant $\underline{\rm K}_{\langle\kappa\rangle}(-)$ if one deletes all the $\kappa$ maps from total K-theory $\underline{\rm K}(-)$ with Bockstein operations.
We say a map $\eta$ is a \emph{$\Lambda$-isomorphism} if  $\eta$ is a group graded isomorphism and preserves the Bockstein operations $\rho,\kappa,\beta$; 
we say a map $\eta$ is a \emph{$\underline{\rm K}_{\langle\kappa\rangle}$-isomorphism} if  $\eta$ is a  graded group isomorphism and preserves the Bockstein operations $\beta,\rho$, 
but not necessarily preserves the map $\kappa$. 

Let $A, B$ be ${\rm C}^*$-algebras. An \emph{ordered scaled $\star$-isomorphism} (for $\star=\Lambda,$ $ \underline{\mathrm{K}}_{\langle\kappa\rangle}$)
$$
(\underline{\rm K}(A),\underline{\rm K}(A)_+,\Sigma A)_{\star}
\cong (\underline{\rm K}(B),\underline{\rm K}(B)_+,\Sigma B)_\star
$$
is a graded group isomorphism that preserves
the positive cone, the scale (and the unit if present), and the relevant operations:
all of $\Lambda$, or all except the $\kappa$ maps,
respectively.

\end{definition}
Our main result is the following.

\begin{theorem}\label{thm:main}
There exist two non-isomorphic unital  AD algebras of real rank zero $E_0$ and
$E_1$ such that
$$
(\underline{\rm K}(E_0),\underline{\rm K}(E_0)_+,[1_{E_0}])_{\underline{\mathrm{K}}_{\langle\kappa\rangle}}
\cong (\underline{\rm K}(E_1),\underline{\rm K}(E_1)_+,[1_{E_1}])_{\underline{\mathrm{K}}_{\langle\kappa\rangle}},
$$
while
$$
(\underline{\rm K}(E_0),\underline{\rm K}(E_0)_+,[1_{E_0}])_{\Lambda}
\ncong (\underline{\rm K}(E_1),\underline{\rm K}(E_1)_+,[1_{E_1}])_{\Lambda}.
$$  
\end{theorem}

We outline the construction, which relies on unbounded $p$-primary torsion.  We first construct
two simple unital $C^*$-algebras of real rank zero $A$ and $B$, together with two
unital homomorphisms $\varphi_0,\varphi_1:A\to B.$
At each fixed $n$, the induced maps
$$
\K(\varphi_0;\mathbb{Z}_n),\K(\varphi_0;\mathbb{Z}_n): \K(A;\mathbb{Z}_n)\to \K(B;\mathbb{Z}_n)
$$
are conjugate by an order-preserving shear.  These shears
preserve the operations $\rho$ and $\beta$, but $\kappa$.  The obstruction is that the shear between mod-$p^r$ groups uses the first $r-1$ coordinates.  

Finally, we construct the desired algebras as continuous field algebras
$$
 E_i=\mathcal{F}[\varphi_i,\varphi_i,\ldots],
 \qquad i=0,1,
$$
in the sense of Dadarlat and Eilers.


The paper is organized as follows. In Section 1, we recall some preliminaries for K-theory and the continuous-field construction. In Section 2, we construct the ordered K-groups and the corresponding $C^*$-algebras. In Section 4, we construct the coefficient system and the two gluing maps, we show that
proves that no coherent
$\kappa$-compatible conjugacy exists. In Section 5, we construct two continuous fields algebras to show our main result.

\section{Preliminaries}\label{sec:preliminaries}
\subsection{Terminologies}
\begin{definition}\rm
   Let $A$ be a unital $\mathrm{C}^*$-algebra. $A$ is said to have stable rank one, written $sr(A)=1$, if the set of invertible elements of $A$ is dense. $A$ is said to have real rank zero, written $rr(A)=0$, if the set of invertible self-adjoint elements is dense in the set $A_{sa}$ of self-adjoint elements of $A$. If $A$ is not unital, let us denote the unitization of $A$ by $\widetilde{A}$. A non-unital $\mathrm{C}^*$-algebra is said to have stable rank one (or real rank zero) if its unitization has stable rank one (or real rank zero).
\end{definition}
\begin{definition}\rm
We shall say a ${\rm C}^*$-algebra is an AD algebra, if it is an inductive limit of finite direct sums of algebras $M_k(\widetilde{\mathbb{I}}_p)$ and $M_k(C(X))$, where $$
\mathbb{I}_p=\{f\in M_p(C_0(0,1]):\,f(1)=\lambda\cdot1_p,\,1_p {\rm \,is\, the\, identity\, of}\, M_p\}
$$ is the Elliott-Thomsen dimension drop interval algebra
and $X$ is one of the following finite connected CW complexes: $\{pt\},~\mathbb{T},~[0, 1].$
\end{definition}

\subsection{The total K-theory}

\begin{definition}\label{def k-total}\rm (\cite[Section 4]{DadarlatGong1997})
For a $\mathrm{C}^{*}$-algebra $A$, one defines the total K-theory of $A$ by
$$
\underline{{\rm K}}(A)=\bigoplus_{n=0}^{\infty} {\rm K}_{*}(A ; \mathbb{Z}_n),
$$
where
$\mathrm{K}_*(A;\mathbb{Z}_0)=\mathrm{K}_*(A)$,
$\mathrm{K}_*(A;\mathbb{Z}_1)=0$ and
 for $n\ge 2$, $\mathrm{K}_* (A;\mathbb{Z}_n)=\mathrm{K}_*(A\otimes C_0(W_n)).
$

It is a $\mathbb{Z}_2 \times \mathbb{Z}^{+}$graded group. For $m\in \mathbb{Z}$, denote $[m]_n:=m+n\mathbb{Z}\in\mathbb{Z}_n.$ It was shown in \cite{Schochet1984} that the coefficient maps
$$
\begin{gathered}
\rho: \mathbb{Z} \rightarrow \mathbb{Z}_n, \quad \rho(1)=[1]_n, \\
\kappa_{mn, m}: \mathbb{Z}_m \rightarrow \mathbb{Z}_{mn}, \quad \kappa_{m n, m}([1]_m)=n[1]_{mn}, \\
\kappa_{n, m n}: \mathbb{Z}_{mn} \rightarrow \mathbb{Z}_n, \quad \kappa_{n, m n}([1]_{mn})=[1]_{n},
\end{gathered}
$$
induce natural transformations
$$
\rho_{n}^{j}: {\rm K}_{j}(A) \rightarrow {\rm K}_{j}(A ; \mathbb{Z}_n),
$$
$$
\kappa_{m n, m}^{j}: {\rm K}_{j}(A ; \mathbb{Z}_m) \rightarrow {\rm K}_{j}(A ; \mathbb{Z}_{mn}),
$$
$$
\kappa_{n, m n}^{j}: {\rm K}_{j}(A ; \mathbb{Z}_{mn}) \rightarrow {\rm K}_{j}(A ; \mathbb{Z}_n) .
$$
The Bockstein map
$$
\beta_{n}^{j}: {\rm K}_{j}(A ; \mathbb{Z}_n) \rightarrow {\rm K}_{j+1}(A)
$$
appears in the six-term exact sequence
$$
\xymatrixcolsep{3pc}
\xymatrix{
{\mathrm{K}_0(A)}  \ar[r]^-{\times n}
& {\mathrm{K}_0(A)}  \ar[r]^-{\rho_{n}^{0}}
& {\mathrm{K}_0(A; \mathbb{Z}_n)} \ar[d]_-{\beta_{n}^{0}}
 \\
{\mathrm{K}_1(A; \mathbb{Z}_n)} \ar[u]_-{\beta_{n}^{1}}
& {\mathrm{K}_1(A)} \ar[l]_-{\rho_{n}^{1}}
& {\mathrm{K}_1(A)} \ar[l]_-{\times n}
}
$$
induced by the cofiber sequence
$$
A \otimes S C_{0}\left(\mathbb{T}\right) \longrightarrow A \otimes C_{0}\left(W_{n}\right) \stackrel{\beta}{\longrightarrow} A \otimes C_{0}\left(\mathbb{T}\right) \stackrel{n}{\longrightarrow} A \otimes C_{0}\left(\mathbb{T}\right),
$$
where $S C_{0}\left(\mathbb{T}\right)$ is the suspension algebra of $C_{0}\left(\mathbb{T}\right)$.

For any $n\in \mathbb{N},\,j=0,1$, if $\psi:\, A\to B$ is a homomorphism, we will denote
 $\underline{\rm K}(\psi):\, \underline{{\rm K}}(A)\to\underline{{\rm K}}(B)$ and ${\rm K}_j (\psi;\mathbb{Z}_n):\, {{\rm K}}_j(A;\mathbb{Z}_n)\to{{\rm K}}_j(B;\mathbb{Z}_n)$ to be the induced maps.

 If $\lambda:\,\underline{{\rm K}}(A)\to\underline{{\rm K}}(B)$ is a graded map, we will denote
$$
\lambda_n^j:\,{\rm K}_j(A;\mathbb{Z}_n)\to{\rm K}_j(B;\mathbb{Z}_n),\,\,n\in \mathbb{N},\,j=0,1
$$
to be the restriction maps. We also write $\lambda_n^*:=\lambda_n^0\oplus\lambda_n^1.$

\end{definition}

We also use the Rosenberg--Schochet's Universal Coefficient Theorem (UCT) (\cite{RS}).
\begin{theorem}[UCT]
\label{thm:UCT}
Let $C$ and $D$ be separable nuclear $C^*$-algebras in the bootstrap
category $\mathcal{N}$.  There is a natural short exact sequence
$$
0
\longrightarrow
\operatorname{Ext}_{\mathbb Z}^{1}
\bigl(
\mathrm K_{*-1}(C),\mathrm K_*(D)
\bigr)
\longrightarrow
\mathrm{KK}(C,D)
\longrightarrow
\operatorname{Hom}_{\mathbb Z}
\bigl(
\mathrm K_*(C),\mathrm K_*(D)
\bigr)
\longrightarrow
0.
$$
In particular, every graded homomorphism
$$
\theta_*:\mathrm K_*(C)\longrightarrow\mathrm K_*(D)
$$
is induced by some class
$$
x\in\mathrm{KK}(C,D).
$$
\end{theorem}

Every class $x\in\mathrm{KK}(C,D)$ induces a homomorphism by the UMCT of Dadarlat-Loring \cite{DadarlatLoring1996B}
$$
\underline{\mathrm K}(x):
\underline{\mathrm K}(C)
\longrightarrow
\underline{\mathrm K}(D)
$$
which commutes with all Bockstein operations
$\rho,\,\beta,\,\kappa.$
Thus $\underline{\mathrm K}(x)$ is automatically $\Lambda$-linear.

\subsection{Dadarlat--Gong classification and existence}

We a weak version of the classification theorem of Dadarlat and Gong (see also \cite{DadarlatLoring1996A}).

\begin{theorem}[Dadarlat--Gong;
\cite{DadarlatGong1997}, Theorem~9.1]
\label{thm:DG-classification}
Let $A$ and $B$ be  AD algebras of real rank zero.
Suppose that
$$
\Gamma:
\bigl(
\underline{\mathrm K}(A),
\underline{\mathrm K}(A)_+,
\Sigma A
\bigr)_{\Lambda}
\longrightarrow
\bigl(
\underline{\mathrm K}(B),
\underline{\mathrm K}(B)_+,
\Sigma B
\bigr)_{\Lambda}
$$
is an ordered scaled $\Lambda$-isomorphism.  Then there exists a
$*$-isomorphism
$
\varphi:\,A\rightarrow B
$
such that
$
\underline{\mathrm K}(\varphi)=\Gamma.
$
\end{theorem}

The morphism version of this result is the existence theorem used below (see also \cite{DadarlatEilers1998}).

\begin{theorem}[Dadarlat--Gong;
\cite{DadarlatGong1997}, Corollary~9.2]
\label{thm:DG-existence}
Let $A$ and $B$ be as in Theorem~\ref{thm:DG-classification}.  Suppose that
$$
\Gamma:
\underline{\mathrm K}(A)
\rightarrow
\underline{\mathrm K}(B)
$$
is a positive $\Lambda$-linear morphism which preserves the scale in the
sense that
$
\Gamma(\Sigma A)\subseteq\Sigma B.
$
Then there exists a $*$-homomorphism
$
\varphi:A\rightarrow B
$
such that
$
\underline{\mathrm K}(\varphi)=\Gamma.
$
If $A$ and $B$ are unital and
$
\Gamma([1_A])=[1_B],
$
then $\varphi$ may be chosen unital.
\end{theorem}

\subsection{Coefficient complexes}

For an abelian group $G$, we denote its subgroup consisting of all the torsion elements of $G$   by $\operatorname{tor}(G)$. Denote by $G[n]$ the subgroup of elements of $G$ annihilated by $n \in \mathbb{N}$.

\begin{definition}\rm 
An ordered abelian group $\left(G, G_{+}\right)$ is called weakly unperforated if it satisfies both of the following conditions:

(i) $G/\operatorname{tor}(G)$ is unperforated, i.e, in the quotient ordered group, $ng \geq 0$ with
$ n\in \mathbb{N}^+$  implies $g\geq0$;

(ii) Given $g\in G_+$, $t\in  \operatorname{tor}(G)$, $n\in \mathbb{N}^+$ and $m\in \mathbb{Z}$ with $ng + mt\in G_+$, then
$t = t' + t''$ for some $t', t''\in\operatorname{tor}(G)$ such that  $mt' = 0$ and $g+ t''\in G_+$.
\end{definition}

Let $G$ be an ordered abelian group, $H$ an abelian group, and $f: G \rightarrow H$ a surjective group homomorphism. Say that $h \in H$ is positive if it is the image of a positive element in $G$. The important and obvious feature of the order on $H$ thus defined (the so-called quotient order) is that every positive element in $H$ lifts to a positive element in $G$.

Recall the finite-coefficient complexes defined by Eilers and Toms.

\begin{definition}[\cite{EilersToms2008}, Definition~2.1]\rm
\label{def:n-coefficient-complex}
Let $n\geq2$.  An \emph{$n$-coefficient complex} is an exact sequence
$$
G_0
\xrightarrow{\rho}
G_n
\xrightarrow{\beta}
G_1
$$
of abelian groups such that, after setting
$$
G_*:=G_0\oplus G_1,
\quad
G_{\square \hspace{-2mm} \tiny{n}}:=G_0\oplus G_n,
$$
the following conditions hold:
(i) $n \mathrm{G}_n=0$.

(ii) $\operatorname{ker} \rho=n \mathrm{G}_0, \operatorname{im} \beta=\mathrm{G}_1[n]$.

(iii) $\mathrm{G}_*$ and $\mathrm{G}_{\square \hspace{-2mm} \tiny{n}}$ are graded ordered groups restricting to the same order on $\mathrm{G}_0$.

(iv) $\mathrm{G}_*$ has the Riesz interpolation property.

(v) $\mathrm{G}_0 \oplus \rho\left(\mathrm{G}_0\right)$ has the quotient order coming from $\operatorname{id}_{G_0} \oplus \rho$.

(vi) $\mathrm{G}_0 \oplus \beta\left(\mathrm{G}_n\right)$ has the quotient order coming from $\mathrm{id}_{G_0} \oplus \beta$.

(vii) $G_0$ is unperforated and $G_*$ is weakly unperforated.

We say that an element $(x, y, z)$ is positive in $\overline{G}$ if and only if
$$
\mathrm{G}_{\square \hspace{-2mm} \tiny{n}} \ni(x, y) \geq 0 \text { and } \mathrm{G}_* \ni(x, z) \geq 0 .
$$

A morphism of $n$-coefficient complexes is a positive triple
$$
(\theta_0,\theta_n,\theta_1)
$$
which commutes with $\rho$ and $\beta$.
\end{definition}

\begin{definition}\rm \label{E def k}
  For real rank zero $\mathrm{AD}$ algebras with $n\cdot$tor\,${\rm K}_1(A)=0$,
$$
\mathbf{K}(A; n): {\rm K}_0(A) \xrightarrow{\rho_n^0} {\rm K}_0(A ; \mathbb{Z}_n) \xrightarrow{\beta_n^0} {\rm K}_1(A)
$$
is a complete invariant \cite{Eilers1996} (both the scale and the order are also required,
where by identifying ${\rm K}_0(A)\oplus{\rm K}_0(A ; \mathbb{Z}_n)\oplus{\rm K}_1(A)$
with its natural image in $\underline{\rm K}(A)$, we say
$$(x_0,x_n,x_1)\in {\rm K}_0(A)\oplus{\rm K}_0(A ; \mathbb{Z}_n)\oplus{\rm K}_1(A)$$
is positive, if $(x_0,x_n,x_1)\in \underline{{\rm K}}(A)_+$).
For any *-homomorphism $\phi:A\to B$, $\phi$ induces $\underline{\rm K}(\phi):\underline{\rm K}(A)\to \underline{\rm K}(B)$, we identify $\mathbf{K}(\phi; n)$ as the triple
$$
({\rm K}_0(\phi),  {\rm K}_0(\phi ; \mathbb{Z}_n), {\rm K}_1(\phi)).
$$
\end{definition}

Eilers and Toms also give a finite-coefficient range theorem.

\begin{theorem}[\cite{EilersToms2008}, Theorem~5.4]
\label{thm:ET-range}
Let $n\geq2$. 
The following conditions are equivalent:
\begin{enumerate}[label=\textup{(\roman*)}]
\item $\mathcal G$ is an $n$-coefficient complex;
\item $\mathcal G$ is an inductive limit of finite direct sums of the
coefficient complexes associated with $\mathbb{C},\,\mathbb I_p^{\sim},\, C(\mathbb{T}),$
and its ordinary graded part has the Riesz interpolation property;
\item there exists an AD algebra of real rank zero $A$ such that
$$
\mathcal G\cong\mathbf K(A;n)
$$
as ordered $n$-coefficient complexes.
\end{enumerate}
\end{theorem}

We use the following simple-unital form of the preceding theorem.

\begin{corollary}
\label{cor:simple-unital-range}
Let
$$
(G,G^+,u)
$$
be a countable simple scaled dimension group, let $H$ be a countable
abelian group annihilated by $n$, and suppose that
$$
G
\xrightarrow{\rho}
G_n
\xrightarrow{\beta}
H
$$
is an $n$-coefficient complex whose order is strict over $G$.
Then there exists a simple unital  AD algebra of real rank zero $D$ such that
$$
\bigl(
\mathrm K_0(D),\mathrm K_0(D)_+,[1_D]
\bigr)
\cong
(G,G^+,u),
\qquad
\mathrm K_1(D)\cong H,
$$
and
$$
\mathbf K_n(D)
\cong
\left(
G\xrightarrow{\rho}G_n\xrightarrow{\beta}H
\right).
$$
\end{corollary}

\begin{proof}
By Theorem~\ref{thm:ET-range}, the coefficient complex is realized by an AD algebra of real rank zero.  In the large-denominator construction used in
the proof of \cite[Theorem~5.4]{EilersToms2008}, the inductive system may
be chosen so that every nonzero positive element eventually becomes full.
Since $(G,G^+)$ is simple, this gives a simple limit algebra.

The scale is chosen so that the prescribed order unit $u$ is represented
by the unit.  Equivalently, one may pass to the full corner corresponding
to $u$.  The resulting algebra is simple, unital, and has the asserted
ordered scaled invariant.
\end{proof}

\begin{remark}
For a C$^*$-algebra $D$ and $n\geq2$, finite-coefficient K-theory gives
the exact sequence
$$
 \K_i(D)
 \xrightarrow{\times n}
 \K_i(D)
 \xrightarrow{\rho_n^i}
 \K_i(D;\mathbb Z_n)
 \xrightarrow{\beta_n^i}
 \K_{i+1}(D)
 \xrightarrow{\times n}
 \K_{i+1}(D),
$$
where the index $i+1$ is interpreted modulo two.  In particular, there
is a natural short exact sequence
\begin{equation}
 0\longrightarrow
 \K_i(D)/n\K_i(D)
 \longrightarrow
 \K_i(D;\mathbb Z_n)
 \xrightarrow{\beta_n^i}
 \K_{i+1}(D)[n]
 \longrightarrow0.
\end{equation}

This short sequence  splits as a sequence
of abelian groups, but the splitting is generally not natural. 
In our construction we use a compatible system of split
coordinates. The existence of the required coefficient complex follows
from the range theorem for AD algebras of real rank zero.
\end{remark}


We shall also use the following standard facts from \cite{DadarlatEilers1999,DL1994aif}.
\begin{proposition}
\label{prop:inductive-limit-facts}
Let
$$
C=\varinjlim(C_r,\psi_r)
$$
be a sequential inductive limit of C$^*$-algebras.
\begin{enumerate}[label=\textup{(\roman*)}]
\item K-theory is continuous:
$$
\mathrm K_i(C)
\cong
\varinjlim
\bigl(
\mathrm K_i(C_r),\mathrm K_i(\psi_r)
\bigr),
\qquad i=0,1.
$$
The same holds for finite-coefficient K-theory.
\item If every $C_r$ has real rank zero, then $C$ has real rank zero.
\item If every $C_r$ is an AD algebra, then $C$ is an AD algebra.
\item If every $C_r$ is simple and every $\psi_r$ is unital, then every
$\psi_r$ is injective and full, and $C$ is simple.
\end{enumerate}
\end{proposition}

\section{Construct the ordered K-groups and algebras}\label{sec:ordered-groups}
\subsection{The ordered K-groups}
\begin{definition}\rm
Fix a prime number $p$.  Let
$$
 T=\mathbb Z(p^\infty)\cong\mathbb Z\left[\frac{1}{p}\right]/\mathbb{Z}
$$
be the Pr$\ddot{\rm u}$fer $p$-group, where $\mathbb Z\left[\frac{1}{p}\right]=\{\frac{l}{p^k}\mid k\in\mathbb N, l\in\mathbb{Z}\}$. 

Let $\{e_j\}_{j=1}^\infty$ and $\{f_j\}_{j=1}^\infty$ be two linearly independent systems.
Denote groups $G_A$ and $G_B$ as follows:
$$
G_A=\{\sum_{j\in J} a_je_j \mid  a_j\in \mathbb{Z} \,\,{\rm and}\ \, J\subset \mathbb{N}^+ \,\,{\rm is\,\, a\,\, finite\,\, subset}\},
$$
$$
G_B=\{\sum_{j\in J} a_jf_j \mid  a_j\in \mathbb{Z} \,\,{\rm and}\ \, J\subset \mathbb{N}^+ \,\,{\rm is\,\, a\,\, finite\,\, subset}\}.
$$
Now choose positive real numbers $\lambda_j$ ($j\geq 1$) such that
$
 1,\lambda_1,\lambda_2,\ldots
$
are linearly independent over $\mathbb Q$.  Define injective
homomorphisms
$$
 \tau_A:G_A\longrightarrow\mathbb R,
 \quad
 \tau_B:G_B\longrightarrow\mathbb R
$$
by
$$\tau_A(e_0)=1,
 \quad
 \tau_A(e_j)=\lambda_j,
 \quad j\geq1,
$$
and
$$
 \tau_B(f_0)=1,
 \quad
 \tau_B(f_j)=\frac{\lambda_j}{p^j},
 \quad j\geq1.
$$
Equip the two groups with the strict positive cones
$$
 G_A^+ = \{0\}\cup\{g\in G_A:\tau_A(g)>0\},
$$
and
$$
 G_B^+=
 \{0\}\cup\{g\in G_B:\tau_B(g)>0\}.
$$
We use $e_0$ and $f_0$ as order units.
\end{definition}
\begin{lemma}\label{lem:rigid-dimension-groups}
The triples
$$
 (G_A,G_A^+,e_0)
 \quad\text{and}\quad
 (G_B,G_B^+,f_0)
$$
are simple scaled dimension groups.  Moreover,
$$
 \Aut(G_A,G_A^+,e_0)
 =
 \Aut(G_B,G_B^+,f_0)
 =
 \{\id\}.
$$
\end{lemma}

\begin{proof}
Both $\tau_A(G_A)$ and $\tau_B(G_B)$ are dense ordered subgroups of
$\mathbb R$.  A dense subgroup of $\mathbb R$, equipped with the induced
order, has the Riesz interpolation property.  Every nonzero positive
element is an order unit, so both ordered groups are simple.

We show that $\tau_A$ is the unique normalized state on $G_A$.  Let
$\sigma$ be a normalized state and let $g\in G_A$.  If
$$
 \frac{a}{b}<\tau_A(g)<\frac{c}{d},
$$
where $a,b,c,d$ are integers with $b,d>0$, then
$$
 bg-ae_0, ce_0-dg\in G_A^+.
$$
Applying $\sigma$ gives
$$
 \frac{a}{b}<\sigma(g)<\frac{c}{d}.
$$
Approximating $\tau_A(g)$ from below and above by rational numbers yields
$$
 \sigma(g)=\tau_A(g).
 $$
Thus $\tau_A$ is the unique normalized state.  The same proof applies
to $\tau_B$.

If $\gamma$ is an order-unit-preserving automorphism of $G_A$, then
uniqueness of the normalized state gives
$$
 \tau_A\circ\gamma=\tau_A.
$$
Since $\tau_A$ is injective, $\gamma=\id$.  The argument for $G_B$ is
similar.
\end{proof}
\begin{definition}\rm\label{def alpha}
Define
$$
 \alpha:G_A\longrightarrow G_B
$$
by
$$
 \alpha(e_0)=f_0,
 \quad
 \alpha(e_j)=p^jf_j,
 \quad j\geq1.
$$
Then
$$
 \tau_B\circ\alpha=\tau_A.
$$
Consequently, $\alpha$ is positive, unital, and injective.
\end{definition}
\subsection{Construction of the algebras} \label{subsec:construction-A}

We now construct the simple unital  AD algebra of real rank zero needed in
Section 5.

For each $r\geq1$, denote
$$
H_r
:=
G_A/p^rG_A\oplus \mathbb{Z}_{p^r}.
$$
Consider the sequence
\begin{equation}
\label{eq:finite-stage-complex}
G_A
\xrightarrow{\rho_r}
H_r
\xrightarrow{\beta_r}
\mathbb{Z}_{p^r},
\end{equation}
where
$$
\rho_r(g)=
(g+p^rG_A,0)\quad
{\rm and}\quad
\beta_r(x,t)=t.
$$

Equip the graded ordinary group
$
G_A\oplus \mathbb{Z}_{p^r}
$
with the strict order
$$
(G_A\oplus \mathbb{Z}_{p^r})^+
=
\{(0,0)\}
\cup
\left\{
(g,t):
g\in G_A^+\setminus\{0\},
\ t\in \mathbb{Z}_{p^r}
\right\}.
$$
Equip $G_A\oplus H_r$
with the corresponding strict order over its $G_A$-component.

\begin{lemma}
\label{lem:finite-complex}
For each $r\geq1$, the sequence
\eqref{eq:finite-stage-complex} is a $p^r$-coefficient complex.
\end{lemma}

\begin{proof}
It is obvious that
$p^rH_r=0$ and $\ker(\rho_r)=p^rG_A.$
Since $\beta_r$ is the projection onto $\mathbb{Z}_{p^r}$,
$
\operatorname{im}(\beta_r)
=
\mathbb{Z}_{p^r}
=
\mathbb{Z}_{p^r}[p^r].
$
Thus the algebraic exactness requirements in
Definition~\ref{def:n-coefficient-complex} hold.

The ordinary graded group $G_A\oplus \mathbb{Z}_{p^r}$, equipped with the strict order
over $G_A$, has the Riesz interpolation property because $G_A$ is a
dimension group and the torsion component is dominated by every nonzero
positive $G_A$-component.  The quotient-order conditions follow directly
from the definitions of $\rho_r$ and $\beta_r$.  Finally, $G_A$ is
unperforated, and the strict graded order makes $G_A\oplus \mathbb{Z}_{p^r}$ weakly
unperforated.
\end{proof}
By Corollary~\ref{cor:simple-unital-range}, for each $r\geq1$, there exists
a simple unital AD algebra of real rank zero $A_r$ such that
\begin{equation}
\bigl(
\mathrm K_0(A_r),
\mathrm K_0(A_r)_+,
[1_{A_r}]
\bigr)
\cong
(G_A,G_A^+,e_0),
\end{equation}
$$
\mathrm K_1(A_r)\cong \mathbb{Z}_{p^r},
$$
and
$$
\mathbf K(A_r;{p^r})
\cong
\left(
G_A
\xrightarrow{\rho_r}
H_r
\xrightarrow{\beta_r}
\mathbb{Z}_{p^r}
\right).
$$

Now we construct the connecting homomorphisms. Recall that
$$
\kappa_{mn, m}: \mathbb{Z}_m \rightarrow \mathbb{Z}_{mn}, \quad \kappa_{m n, m}([1]_m)=n[1]_{mn}.
$$
Define a graded homomorphism
$$
\theta_r:
\mathrm K_0(A_r)\oplus\mathrm K_1(A_r)
\longrightarrow
\mathrm K_0(A_{r+1})\oplus\mathrm K_1(A_{r+1})
$$
by
$$
\theta_r=\operatorname{id}_{G_A}\oplus\kappa_{p^{r+1},p^r}.
$$

Since $A_r$ and $A_{r+1}$ are AD algebras, they are separable nuclear
C$^*$-algebras in the bootstrap category.  By
Theorem~\ref{thm:UCT}, there exists
$$
x_r\in\mathrm{KK}(A_r,A_{r+1})
$$
such that
$$
\mathrm{K}_0(x_r)\oplus \mathrm{K}_1(x_r)=\theta_r.
$$
The class $x_r$ induces a $\Lambda$-linear morphism
$$
\Psi_r
:=
\underline{\mathrm K}(x_r):
\underline{\mathrm K}(A_r)
\longrightarrow
\underline{\mathrm K}(A_{r+1}).
$$
In particular, here, we choose the $x_r$ satisfying
$$
(\Psi_r)_{p^{r+1}}^0={\rm id}_{G_A/p^{r+1}G_A}\oplus \kappa_{p^{r+1},{p^r}}:\, \mathrm K_0(A_{r};\mathbb{Z}_ {p^{r+1}})\to \mathrm K_0(A_{r+1};\mathbb{Z}_{p^{r+1}}).
$$
\begin{lemma}
\label{lem:Gamma-positive}
The morphism $\Psi_r$ is positive and satisfies
$$
\Psi_r([1_{A_r}])=[1_{A_{r+1}}].
$$
\end{lemma}

\begin{proof}
Its ordinary $\mathrm K_0$-component is the identity map on
$$
(G_A,G_A^+,e_0).
$$
It is therefore strictly positive and preserves the distinguished order
unit.

Since $A_r$ and $A_{r+1}$ are simple, every nonzero positive
$\mathrm K_0$-class has full ideal support.  Hence positivity of the
$\Lambda$-morphism is determined by its ordinary $\mathrm K_0$-component.
It follows that $\Psi_r$ preserves the Dadarlat--Gong positive cone and
the scale.
\end{proof}

By Theorem~\ref{thm:DG-existence} (see also a one sided version of Theorem 5.3 of \cite{Eilers1996}), there exists a unital
$*$-homomorphism
$$
\psi_r:A_r\longrightarrow A_{r+1}
$$
such that
$
\underline{\mathrm K}(\psi_r)=\Gamma_r.
$
In particular,
$$
\mathrm K_0(\psi_r)\oplus \mathrm K_1(\psi_r)=\theta_r.
$$
Since $A_r$ is simple and $\psi_r$ is unital, $\psi_r$ is injective.

Now we define
$$
A=\varinjlim \left( A_r,\psi_r\right).
$$

\begin{proposition}
\label{prop:properties-A}
The algebra $A$ is a simple unital  AD algebra of real rank zero, and
$$
\bigl(\mathrm K_0(A),\mathrm K_0(A)_+,[1_A]\bigr)
\cong
(G_A,G_A^+,e_0),
$$
while
$$
\mathrm K_1(A)
\cong
\mathbb Z(p^\infty).
$$
\end{proposition}
\begin{proof}
By Proposition~\ref{prop:inductive-limit-facts}, $A$ is a unital
 AD algebra of real rank zero.  Since every $A_r$ is simple and every
$\psi_r$ is unital, the connecting maps are injective and full.
Therefore $A$ is simple.

At the level of  $\mathrm K_0$-group, all connecting maps are the identity.  Hence
$$
\mathrm K_0(A)
\cong
\varinjlim
\left(
G_A,\operatorname{id}_{G_A}
\right)
\cong
G_A.
$$
The order and the distinguished unit are preserved, so
$$
\bigl(
\mathrm K_0(A),
\mathrm K_0(A)_+,
[1_A]
\bigr)
\cong
(G_A,G_A^+,e_0).
$$

At the level of  $\mathrm{K}_1$-group, the inductive system is
$$
\mathbb Z_p
\xrightarrow{\kappa_{p^2,p}}
\mathbb Z_{p^2}
\xrightarrow{\kappa_{p^3,p^2}}
\mathbb Z_{p^3}
\xrightarrow{}
\cdots\rightarrow \mathbb{Z}(p^\infty).
$$
 Therefore
$
\mathrm K_1(A)
\cong
\mathbb{Z}(p^\infty).
$
\end{proof}

Denoted by $t_r\in\mathrm K_1(A)$ the image of $[1]_{p^r}\in \mathbb{Z}_{p^r}$.  Then
$$
\operatorname{ord}(t_r)=p^r
\quad
{\rm and}\quad
pt_{r+1}=t_r.
$$
Consequently,
$$
\mathrm K_1(A)[p^r]
=
\langle t_r\rangle
\cong
\mathbb Z_{p^r}.
$$

For every $m\geq2$, the coefficient exact sequences give
$$
0
\longrightarrow
G_A/mG_A
\longrightarrow
\mathrm K_0(A;\mathbb Z_m)
\xrightarrow{\beta_m^0}
\mathbb Z(p^\infty)[m]
\longrightarrow
0
$$
and
$$
0
\longrightarrow
\mathbb Z(p^\infty)/m\mathbb Z(p^\infty)
\longrightarrow
\mathrm K_1(A;\mathbb Z_m)
\longrightarrow
G_A[m]
\longrightarrow
0.
$$
Since $G_A$ is torsion free and $\mathbb Z(p^\infty)$ is divisible, then
$$
G_A[m]=0 \quad
{\rm and}\quad
\mathbb Z(p^\infty)/m\mathbb Z(p^\infty)=0.
$$
It follows that
$$
\mathrm K_1(A;\mathbb Z_m)=0.
$$

In particular, for every $s\geq1$,
$$
0
\longrightarrow
G_A/p^sG_A
\longrightarrow
\mathrm K_0(A;\mathbb Z_{p^s})
\xrightarrow{\beta_{p^s}^0}
\langle t_s\rangle
\longrightarrow
0.
$$

\begin{remark}
\label{rem:no-canonical-splitting}
The final short exact sequence splits as a sequence of abelian groups, but
the splitting is not natural. Thus the construction above gives the algebra $A$ and its ordinary K-groups without choosing any splitting.  Suppose we have
$$
\mathrm K_0(A;\mathbb Z_{p^s})
\cong
G_A/p^sG_A\oplus\langle t_s\rangle,
$$
the compatibility of coefficient transformation can be established separately. But we need to go from every coefficient to whole compatibility.
\end{remark}

Finally, Theorem~\ref{thm:DG-classification} shows that the ordered scaled
total K-theory of $A$ determines $A$ up to $*$-isomorphism within the class
of  AD algebras of real rank zero.
\begin{remark}
By the Effros--Handelman--Shen range theorem
\cite{EffrosHandelmanShen1980}, there exists a unital AF algebra $B$
such that
$$
\bigl(
\mathrm{K}_0(B),\mathrm{K}_0(B)_+,[1_B]
\bigr)
\cong
(G_B,G_B^+,u_B).
$$
As $(G_B,G_B^+)$ is simple, our $B$ is chosen simple.
\end{remark}

\section{On the Bockstein operations}

\subsection{The coefficient system}
We first recall a simplified version of the total K-theory by Dadarlat and Loring.

\begin{definition}\rm
  Let $P \subset \mathbb{N}$ be the set consisting of the positive powers of all the primes. Define
  $$
  F_P \underline{\rm K}(A)={\rm K}_*(A) \oplus \bigoplus_{k \in P} {\rm K}_*(A ; \mathbb{Z}_n).$$
\end{definition}
\begin{proposition} [\cite{DadarlatLoring1996B}]\label{Prime K}
Let $A \in \mathcal{N}$ and let $B$ be a $\sigma$-unital $C^*$-algebra. Then the natural restriction map
$$\operatorname{Hom}_{\Lambda}(\underline{\rm K}(A), \underline{\rm K}(B)) \rightarrow \operatorname{Hom}_{\Lambda}\left(F_P \underline{\rm K}(A), F_P \underline{\rm K}(B)\right)$$
is an isomorphism of groups.
\end{proposition}
\begin{remark}
Since every positive integer has a unique prime factorization, by Proposition \ref{Prime K}, suppose we have defined group homomorphisms between finite-coefficient K-theories  at all the coefficients in $P$, then their primary decompositions will give the homomorphisms at all the coefficients. Similarly, if all the primary decompositions is compatible with all coefficient transformations, then implies compatibility with every coefficient-change map.
\end{remark}
We have constructed the simple unital  AD algebra of real rank zero
$A$ and the simple unital AF algebra $B$ satisfying
$$
 \bigl(
 \K_0(A),\K_0(A)_+,[1_A]
 \bigr)
 \cong
 (G_A,G_A^+,e_0),
 \quad
 \K_1(A)\cong T,
$$
and
$$
 \bigl(
 \K_0(B),\K_0(B)_+,[1_B]
 \bigr)
 \cong
 (G_B,G_B^+,f_0),
 \quad
 \K_1(B)=0.
$$

For $r\geq1$, the coefficient exact sequence gives
$$
0
\longrightarrow
G_A/p^rG_A
\xrightarrow{(\rho_A)_{p^r}^0}
\mathrm K_0(A;\mathbb Z_{p^r})
\xrightarrow{(\beta_A)_{p^r}^0}
\langle t_r\rangle
\longrightarrow
0,
$$
and
$$
0
\longrightarrow
G_B/p^rG_B
\xrightarrow{(\rho_B)_{p^r}^0}
\mathrm K_0(B;\mathbb Z_{p^r})
\xrightarrow{(\beta_B)_{p^r}^0}
0.
$$
For convenience, we simply require the algebra $A$ satisfy that
$$
 \K_0(A;\mathbb Z_{p^r})\cong
  G_A/p^rG_A\oplus T[p^r], \quad
 \K_1(A;\mathbb Z_{p^r})=0,
$$
 the natrual Bockstein maps such that
$$
 (\rho_A)_{p^r}^0(a)=(a+p^rG_A,0),\quad
 (\beta_A)_{p^r}^0(x,t)=t,$$
and
$$
 (\rho_B)_{p^r}^0(b)=b+p^rG_B,
 \quad
 (\beta_B)_{p^r}^0=0.
$$

Now the natural maps $\kappa$ are as follows.

At the coefficients with $p$-power, we have
\begin{eqnarray*}
  (\kappa_A)^0_{p^{r+1},p^r}\,:\,\K_0(A;\mathbb Z_{p^r}) &
 \rightarrow &
 \K_0(A;\mathbb Z_{p^{r+1}}) \\
  (a+p^rG_A,t_r) &\mapsto & (pa+p^{r+1}G_A,pt_{r+1})
\end{eqnarray*}
\begin{eqnarray*}
  (\kappa_A)^0_{p^{r},p^{r+1}}\,:\,\K_0(A;\mathbb Z_{p^{r+1}}) &
 \rightarrow &
 \K_0(A;\mathbb Z_{p^{r}}) \\
(a+p^{r+1}G_A,t_{r+1}) &\mapsto & (a+p^{r}G_A,t_r)
\end{eqnarray*}
\begin{eqnarray*}
  (\kappa_B)^0_{p^{r+1},p^r}\,:\,\K_0(B;\mathbb Z_{p^r}) &
 \rightarrow &
 \K_0(B;\mathbb Z_{p^{r+1}}) \\
  b+p^rG_B &\mapsto & pb+p^{r+1}G_B
\end{eqnarray*}
\begin{eqnarray*}
  (\kappa_B)^0_{p^{r},p^{r+1}}\,:\,\K_0(B;\mathbb Z_{p^{r+1}}) &
 \rightarrow &
 \K_0(B;\mathbb Z_{p^{r}}) \\
b+p^{r+1}G_B &\mapsto & b+p^{r}G_A
\end{eqnarray*}
For every prime $q\neq p$, multiplication by $q$ is an automorphism of
$T$.
$$
\xymatrixcolsep{3pc}
\xymatrix{
{\mathrm{K}_0(A)}  \ar[r]^-{\times q^r}
& {\mathrm{K}_0(A)}  \ar[r]^-{\rho_{q^r}^{0}}
& {\mathrm{K}_0(A; \mathbb{Z}_{q^r})} \ar[d]_-{\beta_{q^r}^{0}}
 \\
{\mathrm{K}_1(A; \mathbb{Z}_{q^r})} \ar[u]_-{\beta_{q^r}^{1}}
& \mathbb{Z}(p^\infty) \ar[l]_-{\rho_{q^r}^{1}}
& \mathbb{Z}(p^\infty) \ar[l]_-{\times q^r}
}
$$
Hence
$$
 \K_0(A;\mathbb Z_{q^r})
 \cong
 G_A/q^rG_A,
 \quad
 \K_1(A;\mathbb Z_{q^r})=0.
$$
Similarly, we have
$$
 \K_0(B;\mathbb Z_{q^r})
 \cong
 G_B/q^rG_B,
 \quad
 \K_1(B;\mathbb Z_{q^r})=0.
$$

At the coefficients with  $q$-power ($q\neq p$), we have
\begin{eqnarray*}
  (\kappa_A)^0_{q^{r+1},q^r}\,:\,\K_0(A;\mathbb Z_{q^r}) &
 \rightarrow &
 \K_0(A;\mathbb Z_{q^{r+1}}) \\
  a+q^rG_A &\mapsto & qa+q^{r+1}G_A
\end{eqnarray*}
\begin{eqnarray*}
  (\kappa_A)^0_{q^{r},q^{r+1}}\,:\,\K_0(A;\mathbb Z_{q^{r+1}}) &
 \rightarrow &
 \K_0(A;\mathbb Z_{q^{r}}) \\
a+q^{r+1}G_A, &\mapsto & a+q^{r}G_A
\end{eqnarray*}
\begin{eqnarray*}
  (\kappa_B)^0_{q^{r+1},q^r}\,:\,\K_0(B;\mathbb Z_{q^r}) &
 \rightarrow &
 \K_0(B;\mathbb Z_{q^{r+1}}) \\
  b+q^rG_B &\mapsto & qb+q^{r+1}G_B
\end{eqnarray*}
\begin{eqnarray*}
  (\kappa_B)^0_{q^{r},q^{r+1}}\,:\,\K_0(B;\mathbb Z_{q^{r+1}}) &
 \rightarrow &
 \K_0(B;\mathbb Z_{q^{r}}) \\
b+q^{r+1}G_B &\mapsto & b+q^{r}G_A
\end{eqnarray*}

Note that for any $q\in P$, all the maps
$$
(\kappa_A)^1_{q^{r},q^{r+1}}, (\kappa_A)^1_{q^{r+1},q^{r}}, (\kappa_B)^1_{q^{r},q^{r+1}}, (\kappa_B)^1_{q^{r+1},q^{r}}
$$
are the zero maps.

\subsection{Two total K-theory maps}
\label{Alpha and Phi}
Recall the group homomorphism $\alpha$ defined in \ref{def alpha}. Let
$$
 \overline{\alpha}_{n}:
 G_A/p^rG_A\longrightarrow G_B/nG_B
$$
be the homomorphism induced by $\alpha$. 

For every $r\geq 1$, define
\begin{eqnarray*}
  d_r:T[p^r]&\rightarrow & G_B/p^rG_B \\
  t_r &\mapsto&  \sum_{j=1}^{r-1}p^jf_j+p^rG_B.
\end{eqnarray*}

Now we construct two graded maps
$
 \Phi_0,\Phi_1:
 \underline{\rm K}(A)\longrightarrow\underline{\rm K}(B)
$ such that
$$
(\Phi_0)_{p^r}^0((x,t))=\overline{\alpha}_{p^r}(x), \,\,(\Phi_1)_{p^r}^0((x,t))=\overline{\alpha}_{p^r}(x)+d_r(t)
$$
and
$$
(\Phi_0)_n^j=
(\Phi_1)_n^j=\begin{cases}
               \alpha, & \mbox{if } n=0,\,j=0 \\
               0, & \mbox{if } n=0,\, j=1 \\
               0, & \mbox{if } n=p^r,\,j=1 \\
               0, &  \mbox{if } n=q^r,\,j=1 \\
               \overline{\alpha}_{q^r}, &  \mbox{if } n=q^r,\,j=0,
             \end{cases},
$$
where $q$ is a prime with $q\neq p$.

By Proposition \ref{Prime K}, $\Phi_0$ and $\Phi_1$ can be extended to the map between total K-groups
 through the primary decompositions.

\begin{lemma}\label{lem:kappa-compatibility}
For every $r\geq1$ and $i=0,1$, the following diagrams
$$
\xymatrixcolsep{2pc}
\xymatrix{
{\,\,\mathrm{K}_0(A;\mathbb{Z}_{p^r})\cong G_A/p^rG_A\oplus T[p^r]\,\,}
\ar[r]^-{(\Phi_i)_{p^r}^0}\ar[d]_-{(\kappa_A)_{p^{r+1},p^r}^0}
& {\,\,\mathrm{K}_0(B;\mathbb{Z}_{p^r})\cong G_B/p^rG_B\,\,} \ar[d]_-{{(\kappa_B)_{p^{r+1},p^r}^0}}
  \\
{\,\,\mathrm{K}_0(A;\mathbb{Z}_{p^{r+1}})\cong G_A/p^{r+1}G_A\oplus T[p^{r+1}] \,\,}
\ar[r]_-{(\Phi_i)_{p^{r+1}}^0}
& {\,\,\mathrm{K}_0(B;\mathbb{Z}_{p^{r+1}})\cong G_B/p^{r+1}G_B\,\,}}
$$
and
$$
\xymatrixcolsep{2pc}
\xymatrix{
{\,\,\mathrm{K}_0(A;\mathbb{Z}_{p^r})\cong G_A/p^rG_A\oplus T[p^r]\,\,}
\ar[r]^-{(\Phi_i)_{p^r}^0}
& {\,\,\mathrm{K}_0(B;\mathbb{Z}_{p^r})\cong G_B/p^rG_B\,\,}
  \\
{\,\,\mathrm{K}_0(A;\mathbb{Z}_{p^{r+1}})\cong G_A/p^{r+1}G_A\oplus T[p^{r+1}] \,\,}
\ar[r]_-{(\Phi_i)_{p^{r+1}}^0} \ar[u]_-{(\kappa_A)_{p^{r+1},p^{r}}^0}
& {\,\,\mathrm{K}_0(B;\mathbb{Z}_{p^{r+1}})\cong G_B/p^{r+1}G_B\,\,}\ar[u]_-{{(\kappa_B)_{p^{r+1},p^{r}}^0}}}
$$
are commutative.
\end{lemma}
\begin{proof}
If $i=0$, the commutativity of the diagrams comes from the property of $\alpha$.

If $i=1$, in the first diagram, we choose  $(a+p^rG_A,t_r)\in {\rm K}(A; \mathbb{Z}_{p^r})$, then
\begin{eqnarray*}
  (\kappa_B)_{p^{r+1},p^r}^0\circ (\Phi_1)_{p^r}^0 ((a+p^rG_A,t_r)) &=&   (\kappa_B)_{p^{r+1},p^r}^0 (\overline{\alpha}_{p^r}(a+p^rG_A)+d_r(t_r)) \\
   &=&  (\kappa_B)_{p^{r+1},p^r}^0 (\alpha(a)+ \sum_{j=1}^{r-1}p^{j}f_j+p^rG_B)   \\
   &=& p\alpha(a)+ \sum_{j=2}^{r}p^{j}f_j+p^{r+1}G_B
\end{eqnarray*}
and
\begin{eqnarray*}
  (\Phi_1)_{p^{r+1}}^0 \circ (\kappa_A)_{p^{r+1},p^r}^0(a+p^rG_A,t_r) &=& (\Phi_1)_{p^{r+1}}^0 (pa+p^{r+1}G_A,pt_{r+1}) \\
   &=&  \overline{\alpha}_{p^{r+1}}(pa+p^{r+1}G_A)+d_{r+1}(pt_{r+1})   \\
   &=& p\alpha(a)+ \sum_{j=2}^{r}p^{j}f_j+p^{r+1}G_B.
\end{eqnarray*}
Then the first diagram commutes. Similarly, it is routine to check the commutativity of the second diagram.
\end{proof}

\begin{proposition}\label{prop:Phi-Lambda}
The maps $\Phi_0$ and $\Phi_1$ are positive, unital
$\Lambda$-homomorphisms.
\end{proposition}

\begin{proof}
By Proposition \ref{Prime K}, we only need to check the commutativity of all the coefficients in $P$.

For $i=0,1$ and $q^r\in P$,  consider the  following two kinds of diagrams
$$
\xymatrixcolsep{2.5pc}
\xymatrix{
{\mathrm{K}_0(A)} \ar[d]^-{(\Phi_i)_0^0} \ar[r]^-{\times {q^r}}
& {\mathrm{K}_0(A)} \ar[d]^-{(\Phi_i)_0^0} \ar[r]^-{(\rho_A)_{{q^r}}^0}
& {\mathrm{K}_0(A;\mathbb{Z}_{q^r})}  \ar[r]^-{(\beta_A)_{{q^r}}^0}
\ar[d]^-{(\Phi_i)_{q^r}^0}
& 
 {\mathrm{K}_1(A)} \ar[d]^-{(\Phi_i)_{0}^1} \ar[r]^-{\times {q^r}}
& {\mathrm{K}_1(A)} \ar[d]^-{(\Phi_i)_{0}^1}
 \\
{\mathrm{K}_0(B)} \ar[r]_-{\times {q^r}}
& {\mathrm{K}_0(B)} \ar[r]_-{(\rho_B)_{{q^r}}^0}
& {\mathrm{K}_0(B; \mathbb{Z}_{q^r})} \ar[r]_-{(\beta_B)_{{q^r}}^0}
& 
{\mathrm{K}_{1}(B)} \ar[r]_-{\times {q^r}}
& {\mathrm{K}_{1}(B)}
}
$$
\begin{center}
  ($q_r,0$)-diagram
\end{center}
and
$$
\xymatrixcolsep{2.5pc}
\xymatrix{
{\mathrm{K}_1(A)} \ar[d]^-{(\Phi_i)_0^1} \ar[r]^-{\times {q^r}}
& {\mathrm{K}_1(A)} \ar[d]^-{(\Phi_i)_0^1} \ar[r]^-{(\rho_A)_{{q^r}}^1}
& {\mathrm{K}_1(A;\mathbb{Z}_{q^r})}  \ar[r]^-{(\beta_A)_{{q^r}}^1}
\ar[d]^-{(\Phi_i)_{q^r}^1}
& 
 {\mathrm{K}_0(A)} \ar[d]^-{(\Phi_i)_{0}^0} \ar[r]^-{\times {q^r}}
& {\mathrm{K}_0(A)} \ar[d]^-{(\Phi_i)_{0}^0}
 \\
{\mathrm{K}_1(B)} \ar[r]_-{\times {q^r}}
& {\mathrm{K}_1(B)} \ar[r]_-{(\rho_B)_{{q^r}}^1}
& {\mathrm{K}_1(B; \mathbb{Z}_{q^r})} \ar[r]_-{(\beta_B)_{{q^r}}^1}
& 
{\mathrm{K}_{0}(B)} \ar[r]_-{\times {q^r}}
& {\mathrm{K}_{0}(B)}.
}
$$
\vskip -2mm
\begin{center}
  ($q_r,1$)-diagram
\end{center}
The compatibility of $\Phi_i$ and the operations $\rho, \beta$ follows directly from the commutativity of all these diagrams.

If $q=p$, the ($p_r,0$)-diagram is
$$
\xymatrixcolsep{2pc}
\xymatrix{
{G_A} \ar[d]^-{\alpha} \ar[r]^-{\times {p^r}}
& {G_A} \ar[d]^-{\alpha} \ar[r]^-{(\rho_A)_{{p^r}}^0}
& {G_A/p^rG_A\oplus T[p^r]}  \ar[r]^-{(\beta_A)_{{p^r}}^0}
\ar[d]^-{(\Phi_i)_{p^r}^0}
& 
 {\mathbb{Z}(p^{\infty})} \ar[d]^-{0} \ar[r]^-{\times {p^r}}
& {\mathbb{Z}(p^{\infty})} \ar[d]^-{0}
 \\
{G_B} \ar[r]_-{\times {p^r}}
& {G_B} \ar[r]_-{(\rho_B)_{{p^r}}^0}
& {G_A/p^rG_A\oplus T[p^r]} \ar[r]_-{(\beta_B)_{{p^r}}^0}
& 
{0} \ar[r]_-{\times {p^r}}
& {0}.
}
$$
It is easily seen that
$$
 (\Phi_i)_{p^r}^0\circ(\rho_A)_{p^r}^0(g)
=\overline{\alpha}_{p^r}(g+p^rG_A)=\alpha(g)+p^rG_B,
$$
and
$$
(\rho_B)_{p^r}^0\circ(\Phi_i)_{p^r}^0(g)= (\rho_B)_{p^r}^0(\alpha(g))=\alpha(g)+p^rG_B.
$$
So the ($p_r,0$)-diagram commutes.

Now the ($p_r,1$)-diagram is
$$
\xymatrixcolsep{3pc}
\xymatrix{
{\mathbb{Z}(p^{\infty})} \ar[d]^-{0} \ar[r]^-{\times {p^r}}
& {\mathbb{Z}(p^{\infty})} \ar[d]^-{0} \ar[r]^-{(\rho_A)_{{p^r}}^1}
& {0}  \ar[r]^-{(\beta_A)_{{p^r}}^1}
\ar[d]^-{0}
& 
 {G_A} \ar[d]^-{\alpha} \ar[r]^-{\times {p^r}}
& {G_A} \ar[d]^-{\alpha}
 \\
{0} \ar[r]_-{\times {p^r}}
& {0} \ar[r]_-{(\rho_B)_{{p^r}}^1}
& {0} \ar[r]_-{(\beta_B)_{{p^r}}^1}
& 
{G_B} \ar[r]_-{\times {p^r}}
& {G_B},
}
$$
and it commutes naturally.

If $q\neq p$, the ($q^r,0$)-diagram is
$$
\xymatrixcolsep{3pc}
\xymatrix{
{G_A} \ar[d]^-{\alpha} \ar[r]^-{\times {q^r}}
& {G_A} \ar[d]^-{\alpha} \ar[r]^-{(\rho_A)_{{q^r}}^0}
& {G_A/q^rG_A}  \ar[r]^-{(\beta_A)_{{q^r}}^0}
\ar[d]^-{\overline{\alpha}_{q^r}}
& 
 {\mathbb{Z}(q^{\infty})} \ar[d]^-{0} \ar[r]^-{\times {q^r}}
& {\mathbb{Z}(q^{\infty})} \ar[d]^-{0}
 \\
{G_B} \ar[r]_-{\times {q^r}}
& {G_B} \ar[r]_-{(\rho_B)_{{q^r}}^0}
& {G_B/q^rG_B} \ar[r]_-{(\beta_B)_{{q^r}}^0}
& 
{0} \ar[r]_-{\times {q^r}}
& {0}.
}
$$
Then the fact  ensures the commutativity of the ($q^r,0$)-diagram.

Now the ($q_r,1$)-diagram is
$$
\xymatrixcolsep{3pc}
\xymatrix{
{\mathbb{Z}(p^{\infty})} \ar[d]^-{0} \ar[r]^-{\times {q^r}}
& {\mathbb{Z}(p^{\infty})} \ar[d]^-{0} \ar[r]^-{(\rho_A)_{{q^r}}^1}
& {0}  \ar[r]^-{(\beta_A)_{{q^r}}^1}
\ar[d]^-{0}
& 
 {G_A} \ar[d]^-{\alpha} \ar[r]^-{\times {q^r}}
& {G_A} \ar[d]^-{\alpha}
 \\
{0} \ar[r]_-{\times {q^r}}
& {0} \ar[r]_-{(\rho_B)_{{q^r}}^1}
& {0} \ar[r]_-{(\beta_B)_{{q^r}}^1}
& 
{G_B} \ar[r]_-{\times {q^r}}
& {G_B},
}
$$
and it commutes naturally.

Then we obtain the compatibility with $\rho$ and $\beta$. Now we check
the compatibility with $\kappa$.

Consider the following two kinds of diagrams:
$$
\xymatrixcolsep{2pc}
\xymatrix{
{\,\,\mathrm{K}_0(A;\mathbb{Z}_{q^r})\,\,} \ar[r]^-{(\kappa_A)_{q^{r+1},q^r}^{0}}\ar[d]_-{(\Phi_i)_{q^r}^0}
& {\,\,\mathrm{K}_0(A;\mathbb{Z}_{q^{r+1}})\,\,} \ar[d]_-{ (\Phi_i)_{q^{r+1}}^0}
 \\
{\,\,\mathrm{K}_0(B;\mathbb{Z}_{q^r}) \,\,}\ar[r]_-{(\kappa_B)_{q^{r+1},q^r}^{0}}
& {\,\,\mathrm{K}_0(B;\mathbb{Z}_{q^{r+1}})\,\,}}
\,\,\,\,
{\rm and}
\,\,\,\,
\xymatrixcolsep{2pc}
\xymatrix{
{\,\,\mathrm{K}_0(A;\mathbb{Z}_{q^{r+1}})\,\,} \ar[r]^-{(\kappa_A)_{q^{r},q^{r+1}}^{0}}\ar[d]_-{(\Phi_i)_{q^{r+1}}^0}
& {\,\,\mathrm{K}_0(A;\mathbb{Z}_{q^{r}}).\,\,} \ar[d]_-{ (\Phi_i)_{q^{r}}^0}
 \\
{\,\,\mathrm{K}_0(B;\mathbb{Z}_{q^{r+1}}) \,\,}\ar[r]_-{(\kappa_B)_{q^{r},q^{r+1}}^{0}}
& {\,\,\mathrm{K}_0(B;\mathbb{Z}_{q^{r}})\,\,}}
$$

If $q=p$, then the commutativity of the above diagrams follows from Lemma \ref{lem:kappa-compatibility}.

If $q\neq p$, the above two diagrams become
$$
\xymatrixcolsep{2pc}
\xymatrix{
{\,\,G_A/q^rG_A\,\,} \ar[r]^-{\rm }\ar[d]_-{\overline{\alpha}_{q^r}}
& {\,\,G_A/q^{r+1}G_A\,\,} \ar[d]_-{ \overline{\alpha}_{q^{r+1}}}
 \\
{\,\,G_B/q^rG_B \,\,}\ar[r]_-{\rm }
& {\,\,G_B/q^{r+1}G_B\,\,}}
\,\,\,\,
{\rm and}
\,\,\,\,
\xymatrixcolsep{2pc}
\xymatrix{
{\,\,G_A/q^{r+1}G_A\,\,} \ar[r]^-{\rm }\ar[d]_-{\overline{\alpha}_{q^{r+1}}}
& {\,\,G_A/q^rG_A,\,\,} \ar[d]_-{\overline{\alpha}_{q^r}}
 \\
{\,\,G_B/q^{r+1}G_B\,\,}\ar[r]_-{\rm }
& {\,\,G_B/qG_B\,\,}}
$$
which commute naturally.

Then we obtain the comparability with $\kappa_{q^r}^0$. Since all the maps
$$
(\kappa_A)^1_{q^{r},q^{r+1}}, (\kappa_A)^1_{q^{r+1},q^{r}}, (\kappa_B)^1_{q^{r},q^{r+1}}, (\kappa_B)^1_{q^{r+1},q^{r}}
$$
are the zero maps, it is obvious that the following diagrams commute.
$$
\xymatrixcolsep{2pc}
\xymatrix{
{\,\,\mathrm{K}_1(A;\mathbb{Z}_{q^r})\,\,} \ar[r]^-{0}\ar[d]_-{0}
& {\,\,\mathrm{K}_1(A;\mathbb{Z}_{q^{r+1}})\,\,} \ar[d]_-{ 0}
 \\
{\,\,\mathrm{K}_1(B;\mathbb{Z}_{q^r}) \,\,}\ar[r]_-{0}
& {\,\,\mathrm{K}_1(B;\mathbb{Z}_{q^{r+1}})\,\,}}
,\,
{\rm }
\,\,\,\,
\xymatrixcolsep{2pc}
\xymatrix{
{\,\,\mathrm{K}_1(A;\mathbb{Z}_{q^{r+1}})\,\,} \ar[r]^-{0}\ar[d]_-{0}
& {\,\,\mathrm{K}_1(A;\mathbb{Z}_{q^{r}}).\,\,} \ar[d]_-{0}
 \\
{\,\,\mathrm{K}_1(B;\mathbb{Z}_{q^{r+1}}) \,\,}\ar[r]_-{0}
& {\,\,\mathrm{K}_1(B;\mathbb{Z}_{q^{r}})\,\,}}
$$

Now we conclude that $\Phi_0,\Phi_1$ are $\Lambda$-homomorphism.

Both maps have the same $\K_0$-component map $\alpha$, which is strictly positive and unital. Since $A$ and $B$ are simple, every nonzero positive $\K_0$-class has full ideal support. Hence they both preserve the Dadarlat--Gong order. By the way, the maps also preserve the distinguished unit.
\end{proof}

\subsection{Conjugacy of the morphisms}
\label{sec:local-conjugacy}

For $r\geq1$, define
\begin{eqnarray*}
\ell_r:T[p^r]&\rightarrow& G_A/p^rG_A.\\
  t_r &\mapsto&  \sum_{j=1}^{r-1}e_j+p^rG_A
\end{eqnarray*}
Then
$$
 \overline{\alpha}_r\circ \ell_r(t_r)
 =
 \sum_{j=1}^{r-1}p^jf_j+p^rG_B
 =
 d_r(t_r).
$$

Define a shear map
\begin{eqnarray*}
 U_r:
 G_A/p^rG_A\oplus T[p^r]
 &
 \rightarrow
 &
 G_A/p^rG_A\oplus T[p^r]\\
  (x,t) &\mapsto&  (x-\ell_r(t),t).
\end{eqnarray*}
\begin{lemma}\label{lem:local-shear}
For every $r\geq1$, the map $U_r$ is an automorphism satisfying
$$
 U_r\circ(\rho_A)_{p^r}^0=(\rho_A)_{p^r}^0,
$$
$$
 (\beta_A)_{p^r}^0\circ U_r=(\beta_A)_{p^r}^0,
$$
and
$$
 (\Phi_1)_{p^r}^0\circ U_r=(\Phi_0)_{p^r}^0.
$$
\end{lemma}

\begin{proof}
The inverse of $U_r$ is
$$
 U_r^{-1}(x,t)=(x+\ell_r(t),t).
$$
Moreover,
$$
 U_r(g+p^rG_A,0)
 =
 (g+p^rG_A,0),
$$
so $U_r$ preserves $(\rho_A)_{p^r}^0$.

For the second formula, we see that
$$
 (\beta_A)_{p^r}^0\circ U_r(x,t)
 =
 t
 =
 (\beta_A)_{p^r}^0(x,t).
$$
Finally,
\begin{align*}
 (\Phi_1)_{p^r}^0\circ U_r(x,t)
 &=
 \overline{\alpha}_r(x-\ell_r(t))+d_r(t) \\
 &=
 \overline{\alpha}_r(x)
 -\overline{\alpha}_r(\ell_r(t))
 +d_r(t) \\
 &=
 \overline{\alpha}_r(x) \\
 &=
 (\Phi_0)_{p^r}^0(x,t).
\end{align*}
\end{proof}
\begin{notion}\rm \label{S and kappa}
For  a general positive integer $n$, set
$$
S_n=\begin{cases}
               U_r, & \mbox{if }  p^r\mid n, p^{r+1}\nmid n \\
               id, & otherwise
             \end{cases},
$$
 Then $(\Phi_1)_n^0\circ S_n=(\Phi_0)_n^0$ for every $n\geq2$.  Each $S_n$ preserves the maps $\rho_n^0$ and
$\beta_n^0$.

We note that the family $(S_n)$ does not preserve the operation $\kappa$.  Consider the inclusion
$$
 (\kappa_A)_{p^{r+1},p^r}^0:
 \K_0(A;\mathbb Z_{p^r})
 \longrightarrow
 \K_0(A;\mathbb Z_{p^{r+1}}),
$$
we have
\begin{align*}
 (\kappa_A)_{p^{r+1},p^r}^0\circ S_{p^r}(0,t_r)
 &=
 (\kappa_A)_{p^{r+1},p^r}^0(
 -\sum_{j=1}^{r-1}e_j+p^rG_A,t_r
 ) \\
 &=
(
 -p\sum_{j=1}^{r-1}e_j+p^{r+1}G_A,t_r
),
\end{align*}
whereas
\begin{align*}
 S_{p^{r+1}}\circ(\kappa_A)_{p^{r+1},p^r}^0(0,t_r)
 &=
 S_{p^{r+1}}(0,pt_{r+1}) \\
 &=
(
 -p\sum_{j=1}^{r}e_j+p^{r+1}G_A,t_r
 ).
\end{align*}
Note that
$$
 (pe_r+p^{r+1}G_A,0)\neq 0\,\,\,\, {\rm in}\,\,\,\,G_A/p^{r+1}G_A\oplus T[p^{r+1}].
$$
Therefore,
$$ S_{p^{r+1}}\circ(\kappa_A)_{p^{r+1},p^r}^0\neq (\kappa_A)_{p^{r+1},p^r}^0\circ S_{p^r}.$$
\end{notion}
\subsection{No coherent conjugacy exists}
\label{sec:no-global-conjugacy}

We now prove that no alternative coherent family exists.

\begin{theorem}\label{thm:no-global}
There does not exist ordered scaled $\Lambda$-automorphism
$$
 \Gamma:\underline{\mathrm{K}}(A)\longrightarrow\underline{\mathrm{K}}(A)
$$
such that
$$
 \Phi_1\circ\Gamma=\Phi_0.
$$
\end{theorem}

\begin{proof}
Note that
$$
{\mathrm{K}}_0(A;\mathbb{Z}_m)=
\begin{cases}
G_A, & \mbox{if } m=0 \\
 G_A/q^sG_A\oplus T[p^r], & \mbox{if } m=p^r \\
 G_A/q^sG_A, & \mbox{if } m=q^s, q\neq p
\end{cases}.
$$
and
$$
{\mathrm{K}}_1(A;\mathbb{Z}_m)=
\begin{cases}
\mathbb{Z}(p^{\infty}), & \mbox{if } m=0 \\
 0, & \mbox{if } m\geq 1
\end{cases}.
$$

Suppose that there exists a $\Lambda$-isomorphism  $\Gamma$.
 By Lemma~\ref{lem:rigid-dimension-groups},
$$
 \Gamma_0^0=\id_{G_A}.
 $$
The restriction of $\Gamma$ to
$
 \K_1(A)
$
is an automorphism.  For each $r\geq1$, write
$$
 \Gamma_0^1(t_r)=u_rt_r,
 \qquad
 u_r\in\{1,2,\cdots,p^r-1\}.
$$
Since $pt_{r+1}=t_r$, then we must have
$$u_{r}\equiv u_{r+1} \pmod{p^{r}}.
$$
In particular, $p\nmid u_r$.

Since $\Gamma$ is compatible with  $(\rho_A)_{p^r}^0$ and $(\beta_A)_{p^r}^0$, the following diagram commutes.
\begin{displaymath}
\xymatrixcolsep{4pc}
\xymatrix{
 \mathrm{K}_0(A)  \ar[r]^-{(\rho_A)_{p^r}^0}\ar[d]^-{\rm id}&
 \mathrm{K}_0(A;\mathbb{Z}_{p^r})  \ar[d]^-{\Gamma_{p^r}^0}\ar[r]^-{(\beta_A)_{p^r}^0} & \mathrm{K}_1(A)  \ar[d]^-{\Gamma_0^1} \\
 \mathrm{K}_0(A)  \ar[r]^-{(\rho_A)_{p^r}^0}&
 \mathrm{K}_0(A;\mathbb{Z}_{p^r})  \ar[r]^-{(\beta_A)_{p^r}^0}& \mathrm{K}_1(A).
}
\end{displaymath}
Then for any $(x,t)\in G_A/p^rG_A\oplus T[p^r]$, we have
$$
\Gamma_{p^r}^0(x,t)=
\begin{cases}
  (x,0) & \mbox{if}\,  t=0 \\
  (x+h_r(t),u_rt) & \mbox{if} \, t\neq 0,
\end{cases}
$$
where $h_r$ is a homomorphism
$$
 h_r:T[p^r]\longrightarrow G_A/p^rG_A.
$$

Suppose that $h_r$ has the following form
$$
 h_r(t_r)
 =
 \sum_{j=0}^{\infty}c_{r,j}e_j+p^rG_A,
$$
where only finitely many coefficients $c_{r,j}$ are nonzero modulo
$p^r$.

Since  $\Phi_1\circ\Gamma=\Phi_0$, then the following diagram commutes.
\begin{displaymath}
\xymatrixcolsep{4pc}
\xymatrix{
 \mathrm{K}_0(A;\mathbb{Z}_{p^r})  \ar[d]^-{\Gamma_{p^r}^0}\ar[r]^-{(\Phi_0)_{p^r}^0} & \mathrm{K}_0(B;\mathbb{Z}_{p^r}) . \\
 \mathrm{K}_0(A;\mathbb{Z}_{p^r})  \ar[ru]_-{(\Phi_1)_{p^r}^0}&
}
\end{displaymath}
Recall that (see \ref{Alpha and Phi})
$$
(\Phi_0)_{p^r}^0((x,t))=\overline{\alpha}_r(x), \,\,(\Phi_1)_{p^r}^0((x,t))=\overline{\alpha}_r(x)+d_r(t)
$$
Then for $(0,t_r)\in G_A/p^rG_A\oplus T[p^r]$, we have
$$
   (\Phi_1)_{p^r}^0\circ \Gamma_{p^r}^0(0,t_r) = (\Phi_0)_{p^r}^0 (0,t_r),
$$
which is
$$
 \overline{\alpha}_r(h_r(t_r))+d_r(u_rt_r) = 0,
$$
i.e.,
$$
\sum_{j=0}^{\infty}p^jc_{r,j}e_j+\sum_{j=1}^{r-1}u_rp^jf_j+p^rG_B=0.
$$
Now we have
$$
 p^jc_{r,j}+u_rp^j
 \equiv0\pmod{p^r},\quad j=1,2,\cdots, r-1.
$$
Thus
$$
 c_{r,j}
 \equiv
 -u_r
 \pmod{p^{r-j}}.
$$
In particular,
$$
 c_{r,j}\not\equiv0\pmod p,
 \quad j=1,2,\cdots, r-1.
$$

Recall that
\begin{eqnarray*}
  (\kappa_A)^0_{p,p^{r}}\,:\,\K_0(A;\mathbb Z_{p^{r}}) &
 \rightarrow &
 \K_0(A;\mathbb Z_{p}). \\
(a+p^{r}G_A,t_{r}) &\mapsto & (a+pG_A,t_1)
\end{eqnarray*}
Since $\Gamma$ preserves the operation $\kappa$, the following diagram commutes.
\begin{displaymath}
\xymatrixcolsep{4pc}
\xymatrix{
 \mathrm{K}_0(A;\mathbb{Z}_{p^r})  \ar[d]^-{\Gamma_{p^r}^0}\ar[r]^-{(\kappa_A)^0_{p,p^{r}}} &  \mathrm{K}_0(A;\mathbb{Z}_{p})  \ar[d]^-{\Gamma_{p}^0} \\
 \mathrm{K}_0(A;\mathbb{Z}_{p^r})  \ar[r]^-{(\kappa_A)^0_{p,p^{r}}}&  \mathrm{K}_0(A;\mathbb{Z}_{p}) .
}
\end{displaymath}
Then for $(0,t_r)\in G_A/p^rG_A\oplus T[p^r]$, we have
$$
(\kappa_A)_{p,p^r}^0\circ\Gamma_{p^r}^0(0,t_r) = \Gamma_{p}^0\circ (\kappa_A)_{p,p^r}^0(0,t_r)
$$
$$
 (\kappa_A)_{p,p^r}^0(h_r(t_r),u_rt_r) =  \Gamma_{p}^0(0,t_1),
$$
$$
 (\kappa_A)_{p,p^r}^0(\sum_{j=0}^{\infty}c_{r,j}e_j+p^rG_A,u_rt_r) =  \Gamma_{p}^0(0,t_1),
$$
$$
  (\sum_{j=0}^{\infty}c_{r,j}e_j+pG_A,u_rt_1) = (\sum_{j=0}^{\infty}c_{1,j}e_j+pG_A,u_1t_1).
$$
Now we have
$$
 c_{r,j}
 \equiv
 c_{1,j}
 \pmod{p}.
$$
Since
$$
 c_{r,j}\not\equiv 0\pmod p,
 \quad j=1,2,\cdots, r-1.
$$
Then for any $j\geq 1$, let $r=2j$, then
$$
c_{1,j}\equiv c_{2j,j}
 \not\equiv 0
 \pmod{p}.
$$
This contradicts to the fact that $\sum_{j=0}^{\infty}c_{1,j}e_j$ contains only finite-support vectors.

Hence, no such $\Gamma$ exists.
\end{proof}

\section{On the necessity of $\kappa$}

\subsection{Continuous field ${C}^*$-algebras}

\begin{definition}\label{construction F}\rm
We write $\alpha X$ for the one point (Alexandroff) compactification of a locally compact space $X$,
and denote the point at infinity by $\infty_X$. 
Suppose that $\{\phi_m\}: A \to B$ is a family of *-homomorphisms, we denote by
$$\mathcal{F}[(\phi_m)]\quad {\rm or}\quad \mathcal{F}[\phi_1, \phi_2, \phi_3, \phi_4,\cdots]$$
the C*-algebra
$$\{(a,(b_m))\in A\oplus \prod_{m=1}^\infty B:\, \|b_m-\phi_m(a)\|\to 0\}.
$$

Note that $\mathcal{F}[(\phi_m)]$ is an inductive limit of $C^*$-algebras of the form $A\oplus \bigoplus_1^{m}B$ using bonding maps of the form
$$
\chi_m(a,b_1,b_2,\cdots,b_m)=(a,b_1,b_2,\cdots,b_m,\phi_{m+1}(a)).
$$
\end{definition}
The following contains some result of Dadarlat and Eilers.
\begin{lemma}[\cite{DadarlatEilers1998}]\label{lem:continuous-field}
Let $A$ and $B$ be separable $C^*$-algebras.
\begin{enumerate}[label=\textup{(\roman*)}]
\item There is a split exact sequence
$$
 0\longrightarrow
 \bigoplus_{m=1}^{\infty}B
 \longrightarrow
\mathcal{F}[(\psi_m)]
 \longrightarrow
 A
 \longrightarrow0.
$$

\item The algebra $\mathcal{F}[(\psi_m)]$ is the inductive limit of
$$
 A\oplus\bigoplus_{m=1}^{k}B,
 \qquad k\geq1,
$$
with connecting maps
$$
 (a,b_1,\ldots,b_k)
 \longmapsto
 (a,b_1,\ldots,b_k,\psi_{k+1}(a)).
$$

\item If $A$ and $B$ are  AD algebras of real rank zero, then
$[\mathcal{F}(\psi_m)]$ is an  AD algebra of real rank zero.

\item If $A$ and $B$ are simple, then
$$
 \Prim(\mathcal{F}[(\psi_m)])\cong\alpha\mathbb N.
$$
\end{enumerate}
\end{lemma}

Recall that $\Phi_0,\Phi_1$ are $\Lambda$-homomorphisms (Proposition \ref{prop:Phi-Lambda}). By Theorem~\ref{thm:DG-existence}, there are unital
$*$-homomorphisms
$$
 \varphi_0,\varphi_1:A\longrightarrow B
$$
such that
$$
 \underline{\mathrm{K}}(\varphi_i)=\Phi_i,
 \qquad i=0,1.
$$
\begin{definition}\rm
By Theorem~\ref{thm:DG-existence}, there are unital
$*$-homomorphisms
$$
 \varphi_0,\varphi_1:A\longrightarrow B
$$
such that
$$
 \underline{\mathrm{K}}(\varphi_i)=\Phi_i,
 \qquad i=0,1.
$$
Define the $C^*$-algebras
$$
 E_0
 =
 \mathcal{F}[\varphi_0,\varphi_0,\varphi_0,\ldots]
\quad{\rm and}\quad
 E_1
 =
 \mathcal{F}[\varphi_1,\varphi_1,\varphi_1,\ldots].
$$
\end{definition}
By Lemma~\ref{lem:continuous-field}, both $E_0$ and $E_1$ are unital
 AD algebras of real rank zero. They fit into split extensions
$$
 0\longrightarrow
 \bigoplus_{m=1}^{\infty}B
 \longrightarrow
 E_i
 \longrightarrow
 A
 \longrightarrow0.
$$

For $k\geq1$, set
$$
 E_i^{(k)}
 =
 A\oplus\bigoplus_{m=1}^{k}B.
$$
The connecting map
$$
 \chi_k^{(i)}:
 E_i^{(k)}
 \longrightarrow
 E_i^{(k+1)}
$$
is
$$
 \chi_k^{(i)}(a,b_1,\ldots,b_k)
 =
 (a,b_1,\ldots,b_k,\varphi_i(a)).
$$
Then
$$
 E_i
 \cong
 \varinjlim
 \left(
 E_i^{(k)},\chi_k^{(i)}
 \right).
$$
\begin{definition}\rm
Let us define a  family of group homomorphisms
$$
(\Theta^{(k)})^{*}_n:
{\rm K}_*(E_0^{(k)};\mathbb Z_n)\longrightarrow{\rm K}_*(E_1^{(k)};\mathbb Z_n).
$$
 such that
$$
(\Theta^{(k)})^{*}_n=\begin{cases}
                 \id_{\K_*(A)}\oplus \bigoplus_{m=1}^k \id_{\K_*(B)}, & \mbox{if } n=0 \\
               0, & \mbox{if } n=1 \\
                S_n\oplus
\bigoplus_{m=1}^k  \id_{\K_*(B;\mathbb Z_n)} & \mbox{if } n\geq 2
             \end{cases}.
             $$
Passing to inductive limits, we have a family of isomorphisms
$$
(\Theta)^{*}_n:
 \K_*(E_0;\mathbb Z_n)
 \longrightarrow
 \K_*(E_1;\mathbb Z_n).
$$

Let
$$
\Theta^{(k)}:=\bigoplus_{n=0}^{\infty}(\Theta^{(k)})^{*}_n\quad
{\rm and}\quad
\Theta:=\bigoplus_{n=0}^{\infty}(\Theta)^{*}_n,$$ then we have graded homomorphsims
$$
 \Theta^{(k)}:
 \underline{\rm K}(E_0^{(k)})\longrightarrow\underline{\rm K}(E_1^{(k)})
\quad{\rm and }\quad
 \Theta:
 \underline{\mathrm{K}}(E_0)
 \longrightarrow
 \underline{\mathrm{K}}(E_1).
$$

\end{definition}
\begin{lemma}\label{lem:stage-commutativity}
For every $k\geq1$, we have
$$
(\Theta^{(k+1)})^{*}_n\circ
 \K_*(\chi_k^{(0)};\mathbb Z_n)
 =
 \K_*(\chi_k^{(1)};\mathbb Z_n)
 \circ(\Theta^{(k)})^{*}_n.
$$
\end{lemma}

\begin{proof}
The maps agree on the first $k$ copies of the $B$-summand.  On the
new $B$-summand, the desired equality is precisely
$$
 (\Phi_1)_n^0\circ S_n=(\Phi_0)_n^0.
$$
If $n=0$, then $\Phi_0$ and $\Phi_1$ agree, and all stage maps
are the identity on their existing summands.
\end{proof}

\begin{proposition}\label{prop:reduced-invariant-iso}
The map $\Theta$ is an ordered scaled
$\underline{\mathrm{K}}_{\langle\kappa\rangle}$-isomorphism:
$$
 \bigl(
 \underline{\mathrm{K}}(E_0),\underline{\mathrm{K}}(E_0)_+,[1_{E_0}]
 \bigr)_{ \underline{\mathrm{K}}_{\langle\kappa\rangle}}
 \cong
 \bigl(
 \underline{\mathrm{K}}(E_1),\underline{\mathrm{K}}(E_1)_+,[1_{E_1}]
 \bigr)_{ \underline{\mathrm{K}}_{\langle\kappa\rangle}}.
$$
\end{proposition}

\begin{proof}
Each $S_n$ preserves $\rho_n^0$ and $\beta_n^0$, and the identity maps on
the $B$-summands clearly do so.  Hence $\Theta$ preserves the operations  $\rho,\beta$ retained in $\underline{\mathrm{K}}_{\langle\kappa\rangle}$.

At every finite stage, $(\Theta^{(k)})_{n}^{*}$ fixes the ordinary $\K_0$-group.
The stage algebras are finite direct sums of the simple algebras $A$
and $B$.  Therefore the ideal generated by a projection is determined
by the direct summands on which its ordinary $\K_0$-class is nonzero.
Since $(\Theta^{(k)})_{n}^{*}$  fixes the ordinary $\K_0$-component and each direct
summand, it preserves the Dadarlat--Gong positive cone and the ideal
support of every element.

These properties pass to the inductive limit.  The unit is fixed on
ordinary $\K_0$, so $\Theta$ also preserves the scale.
\end{proof}

The isomorphism $\Theta$ is not a $\Lambda$-isomorphism, since its
$p^r$-components contain the shears $S_r$, and \ref{S and kappa} shows that these do not commute with the
coefficient inclusions.

\subsection{ $E_0$ and $E_1$ are not isomorphic}
\label{sec:nonisomorphism}

\begin{lemma}\label{lem:fiberwise-form}
Suppose that
$$
 \Xi:E_0\longrightarrow E_1
$$
is a $*$-isomorphism.  After composing $\Xi$ with a coordinate
permutation automorphism of $E_1$, there exist automorphisms
$$
 \xi\in\Aut(A),
 \qquad
 \xi_m\in\Aut(B),
$$
such that
$$
 \Xi(a,(b_m))
 =
 \bigl(\xi(a),(\xi_m(b_m))\bigr)
$$
and
\begin{equation}\label{eq:asymptotic-fiber-relation}
 \|
 \xi_m\varphi_0(a)-\varphi_1\xi(a)
 \|
 \longrightarrow0
\end{equation}
for every $a\in A$.
\end{lemma}

\begin{proof}
The isomorphism $\Xi$ induces a homeomorphism of
$$
 \Prim(E_i)\cong\alpha\mathbb N.
$$
The point at infinity is the unique non-isolated point, so it must be
fixed.  The isolated points are permuted.  Since both fields use a
constant gluing map, every permutation $\sigma$ of the isolated points
is implemented by the coordinate-permutation automorphism
$$
 (a,(b_m))
 \longmapsto
 (a,(b_{\sigma^{-1}(m)})).
$$
After composing with such an automorphism, we may assume that $\Xi$
fixes every primitive ideal.

It follows that $\Xi$ induces an automorphism $\xi$ on the quotient
fiber $A$ and automorphisms $\xi_m$ on the isolated fibers $B$.  Applying
$\Xi$ to
$$
 (a,(\varphi_0(a))_{m=1}^{\infty})\in E_0
$$
and using the defining convergence condition for $E_1$ gives
\eqref{eq:asymptotic-fiber-relation}.
\end{proof}

\begin{lemma}\label{lem:B-automorphisms-trivial}
For every automorphism $\eta\in\Aut(B)$,
$$
 \underline{\mathrm{K}}(\eta)=\id_{\underline{\mathrm{K}}(B)}.
$$
\end{lemma}

\begin{proof}
By Lemma~\ref{lem:rigid-dimension-groups},
$
 \K_0(\eta)=\id_{G_B}.
$
Since $B$ is AF, we have
$$
 \K_1(B)=0,\quad
 \K_0(B;\mathbb Z_n)
 \cong
 G_B/nG_B,
 \quad
 \K_1(B;\mathbb Z_n)=0.
$$
Naturality of $\rho_n^0$ and the surjectivity of
$$
 \rho_n^0:G_B\longrightarrow G_B/nG_B
$$
therefore imply
$$
 \K_0(\eta;\mathbb Z_n)=\id
$$
for every $n$.  Hence $\eta$ induces the identity on full total
K-theory.
\end{proof}
We list the following stable relation result, which we will use later.
\begin{theorem}{\rm (}\cite[Theorem 3.6]{Loring1993}{\rm)}\label{sr}
Let $R$ be a set of homotopy liftable relations. There exists
$\varepsilon_0>0$ such that if
$$
  \varphi,\psi:\,C^*\langle G\mid R\rangle\longrightarrow B
$$
are two homomorphisms to any $C^*$-algebra $B$ with
$$
  \lVert\varphi(g_i)-\psi(g_i)\rVert\leq\varepsilon_0,
  \qquad i=1,\ldots,l,
$$
then $\varphi$ and $\psi$ are homotopic. If $R$ is strongly stable, and
$\eta\geq0$ is given, $\varepsilon_0\geq0$ may be chosen so that the
homotopy $\theta_t$ from $\varphi$ to $\psi$ can satisfy
$$
  \lVert\theta_t(g_i)-\varphi(g_i)\rVert\leq\eta.
$$
\end{theorem}

\begin{theorem}\label{thm:E-not-isomorphic}
The algebras $E_0$ and $E_1$ are not isomorphic.
\end{theorem}

\begin{proof}
Suppose that $\Xi:E_0\longrightarrow E_1$
is an isomorphism.  By Lemma~\ref{lem:fiberwise-form}, there exist
automorphisms $\xi$ and $\xi_m$ satisfying
\eqref{eq:asymptotic-fiber-relation}.

Fix an integer $n$.  Consider the unsuspended E-theory in \cite{DL1994}, we have
$$\K_*(A;\mathbb Z_n)\cong \K\K(\mathbb{I}_n,A\otimes C(\mathbb T))=[\mathbb{I}_n,A\otimes C(\mathbb T)\otimes \mathcal{K}].$$
Then point-norm convergence \eqref{eq:asymptotic-fiber-relation} and Theorem \ref{sr} imply
$$
 \K_*(\xi_m;\mathbb Z_n)\circ(\Phi_0)_n^*
 =
 (\Phi_1)_n^*\circ
 \K_*(\xi;\mathbb Z_n).
$$
 for all sufficiently large $m$.
By Lemma~\ref{lem:B-automorphisms-trivial},
$$
 \K_*(\xi_m;\mathbb Z_n)=\id.
$$
Hence
$$
 (\Phi_1)_n^*\circ
 \K_*(\xi;\mathbb Z_n)
 =
 (\Phi_0)_n^*
$$
for every $n$.  The same argument on ordinary K-theory gives
$$
 (\Phi_1)_0^*\circ
 \K_*(\xi)
 =
 (\Phi_0)_0^*.
$$

Since $\xi$ is a $*$-automorphism, the map
$
 \Gamma=\underline{\mathrm{K}}(\xi)
$
is an ordered scaled $\Lambda$-automorphism of $\underline{\mathrm{K}}(A)$.
The preceding equalities give
$
 \Phi_1\Gamma=\Phi_0,
$
contradicting Theorem~\ref{thm:no-global}.  Therefore
$
 E_0\not\cong E_1.
$
\end{proof}

\begin{proof}[Proof of Theorem~\ref{thm:main}]
Proposition~\ref{prop:reduced-invariant-iso} gives
$$
 \bigl(
 \underline{\mathrm{K}}(E_0),\underline{\mathrm{K}}(E_0)_+,[1_{E_0}]
 \bigr)_{\underline{\mathrm{K}}_{\langle\kappa\rangle}}
 \cong
 \bigl(
 \underline{\mathrm{K}}(E_1),\underline{\mathrm{K}}(E_1)_+,[1_{E_1}]
 \bigr)_{\underline{\mathrm{K}}_{\langle\kappa\rangle}}.
$$
Theorem~\ref{thm:E-not-isomorphic} gives
$$
 E_0\not\cong E_1.
$$
Since ordered scaled total K-theory with the full $\Lambda$-module
structure classifies  AD algebras of real rank zero, the full invariants
cannot be isomorphic.  
\end{proof}

\section*{Acknowledgements}
The authors gratefully acknowledge ChatGPT 5.6 sol for its initial suggestion on constructing the shear. However, the arguments and proof outlines generated by the tool were incorrect. All mathematical content was therefore carefully reviewed and substantially revised by the authors, who bear full responsibility for the final version.

\end{document}